\documentclass[a4paper,11pt]{amsart}
\usepackage{hyperref,latexsym}
\usepackage{enumerate}
\usepackage{graphicx}
\usepackage{color}
\usepackage{amsmath}

\theoremstyle{plain}
\newtheorem{theorem}{Theorem}[section]
\newtheorem{lemma}[theorem]{Lemma}
\newtheorem{corollary}[theorem]{Corollary}
\newtheorem{proposition}[theorem]{Proposition}
\theoremstyle{definition}

\theoremstyle{remark}
\newtheorem{remark}{Remark}

\begin{document}

\title[]
      {The Birkhoff-Poritsky conjecture for centrally-symmetric billiard tables}

\date{30 July 2020}
\author{Misha Bialy}
\address{School of Mathematical Sciences, Raymond and Beverly Sackler Faculty of Exact Sciences, Tel Aviv University,
Israel} 
\email{bialy@post.tau.ac.il}
\thanks{MB was partially supported by ISF grant 580/20.}

\author{Andrey E. Mironov}
\address{Novosibirsk State University, Novosibirsk, Russia, and Sobolev Institute
	of Mathematics of the Siberian Branch of the Russian Academy of Sciences,
	Novosibirsk, Russia.
	}
\email{mironov@math.nsc.ru}
\thanks{
	AEM was supported by Mathematical Center in Akademgorodok under agreement
	No 075-2019-1675 with the Ministry of Science and Higher Education of the Russian
	Federation.
	}

%\subjclass[2000]{} 
%\keywords{}

\begin{abstract} In this paper we prove the Birkhoff-Poritsky conjecture for centrally-symmetric $C^2$-smooth convex planar billiards. We assume that the domain $\mathcal A$ between the invariant curve of $4$-periodic orbits and the boundary of the phase cylinder is foliated by $C^0$-invariant curves.  Under this assumption we prove that the billiard curve is an ellipse. 
	For the original Birkhoff-Poritsky formulation we show that if a neighborhood of the boundary of billiard domain has
	a $C^1$-smooth foliation by convex caustics of rotation numbers in the interval (0; 1/4]
	then the boundary curve is an ellipse. In the language of first integrals one can assert that 
{if the billiard inside a centrally-symmetric $C^2$-smooth convex curve $\gamma$ admits a $C^1$-smooth first integral with non-vanishing gradient on $\mathcal A$, then the curve $\gamma$ is an ellipse. }

	The main ingredients of the proof are : (1) the non-standard generating function for convex billiards discovered in \cite{BM}, \cite{B};  (2) the
	remarkable structure of the invariant curve consisting of $4$-periodic orbits;  and (3) the integral-geometry approach initiated in
	\cite{B0}, \cite{B1} for rigidity results of circular billiards. Surprisingly, we establish a Hopf-type rigidity for billiard in ellipse.

\end{abstract}

\maketitle

%%%%%%%%%%%%%%%%%%%%%%%%%%%%%%%%%%%%%%%%%%%%%%%%%%%%%%%%%%%%%%%%%%%%%%%%%%

\section{Introduction}
\subsection{Formulation of the conjecture}
It was G. D. Birkhoff who introduced and studied systematically billiards in strictly convex planar domains.
His results on periodic orbits and invariant curves for twist maps are very important for the modern theory, including  this paper. 
A conjecture attributed to Birkhoff was formulated by Hillel Poritsky \cite{Poritsky}, asking if the only integrable convex billiards are ellipses. When one formulates this conjecture  it is very important to define what "integrability" means. 
One of the strongest form of the conjecture uses the assumption of  the existence of a neighborhood of the boundary curve foliated by caustics. Even
more generally (see \cite{katok}) one can assume that the ends of the phase cylinder are foliated by rotational (=winding once around the cylinder and simple) invariant curves. We refer to the books \cite{Tab}, \cite{KT} for the general theory of billiards.

Several important results were obtained towards the positive resolution of this conjecture. 
Let us mention some of them. An approach by Delshams and Ramirez-Ros \cite{delshams} is based on the proof of splitting of separatrices  for perturbations of elliptic billiard.
Another approach, suggested by classical mechanics, is to look for convex billiards admitting an additional conservation law, which is polynomial in momenta.
Here important  results were obtained by Bolotin \cite{bolotin}, Tabachnikov \cite{T}, Bialy and Mironov \cite{BM}, Glutsyuk and Shustin \cite{GS}  and Glutsyuk \cite{G}.
Innami \cite{Innami}, using Aubry-Mather theory,  and later Arnold and Bialy  \cite{AB}, using purely geometric ideas, proved that the billiard must be an ellipse if there exists a sequence of convex caustics with rotation numbers tending to $1/2$.
Recently, the series of papers by Avila, Kaloshin and De Simoi \cite{AKS} and by Kaloshin and Sorrentino,  \cite{K-S}, \cite{K-S1} were published where {\it local }Birkhoff conjecture is proved.
Here locality means  a suitable neighborhood of ellipses in a functional space.

\subsection{The approach based on total integrability} 
Our approach in this paper is based on the rigidity of Total Integrability (a term suggested by A. Knauf, meaning existence of a foliation of the whole phase space by invariant tori). 
Analyzing E. Hopf proof on geodesic flows with no conjugate points it was understood \cite {B0}, \cite{B1} that Hopf's result can be generalized to  
convex billiards.  As a corollary one gets  that the only {\it totally integrable} billiard is circular. Thus, the circular billiard is analogous to the geodesic flow on the flat torus. 

It is amazing that elliptic billiards also can be distinguished as a rigid objects in both the variational and total integrability settings. The difference with circular billiards appears in the requirement of total integrability on a  certain neighborhood of the boundary of the phase cylinder.
Notice, that in \cite{B} a result of rigidity for  total integrability on a  part of the phase cylinder was obtained for Gutkin billiards.

Our approach in this paper is restricted to the case of centrally-symmetric curves. But our smoothness assumptions are minimal (continuity and even less).  Also the curve is not assumed to be close to an ellipse.

Let us indicate the main tools of our approach. The first ingredient is a non-standard  generating function for the billiard ball map leading to the twist property with respect to another vertical foliation of the phase cylinder $\mathbf A$. 

The second ingredient is the remarkable structure of the invariant curve consisting of 4-periodic orbits. 

The third ingredient is the use of integral geometry essentially as for the circular billiards, but now using the non-standard generating function, and carefully chosen weights. This is a new element in the integral geometry tool.

{ Interestingly, the integral inequalities which we get on this way are much more complicated than in the circular case, but mysteriously can be simplified to the Wirtinger inequality. These calculations are performed in Subsection \ref{reduction}.
	Before performing these calculations we checked the inequalities numerically.
	 It would be very interesting to understand in a more conceptual way these simplifications.  }

\subsection{Main result and discussion}
Throughout this paper, we shall denote by $\gamma\subset\mathbf R^2$ a simple closed centrally-symmetric $C^2$ curve of positive curvature. We shall fix the counterclockwise orientation on $\gamma$. Let $\mathbf A$ be the phase cylinder of the billiard ball map $T$, i.e. the space of all oriented lines intersecting $\gamma$.
Our main result is the following:
\begin{theorem}\label{main1}
	Suppose that the billiard ball map $T$ of $\gamma$  has a continuous rotational (=winding once around the cylinder and simple) invariant curve $\alpha\subset \mathbf A$  of rotation number $1/4$, consisting of $4$-periodic orbits.
	Let  $\mathcal A\subset \mathbf A$ be the domain between the curve $\alpha$ and the boundary  $\{\delta=0\}$ of the phase cylinder, where $\delta$ denotes the incoming angle of the line. 
	
	If  $\mathcal A$ is
	foliated by continuous rotational invariant curves, then $\gamma $ is an ellipse.
\end{theorem}
\begin{corollary}
If a neighborhood of the boundary of the billiard domain is $C^1$-foliated by convex caustics of rotation numbers $(0, 1/4]$, then $\gamma$ is an ellipse.
\end{corollary}
\begin{corollary}\label{F}
If the billiard ball map $T:\mathbf A\rightarrow\mathbf A$ has a $C^1$-first integral $F$ with non-vanishing gradient on $\mathcal A\cup\alpha$, then $\gamma$ is an ellipse.
\end{corollary}

These corollaries follow immediately from Theorem \ref{main1} using the following folkloric fact. 
Every rotational invariant curve which is a leaf of an invariant  $C^1$-foliation always inherits an absolutely continuous invariant measure. Thus, if this invariant curve has a rational rotation number, then all orbits on it are periodic.

Several remarks and questions arise naturally:
  
 1. An interesting question is whether one can consider a smaller region than $\mathcal A$ for the approach of Total integrability suggested in this paper. For example, can one replace $\alpha$ with the invariant curve of period $8, 16, $ etc..?
 
 2. We hope that the central-symmetry restriction can be relaxed.
 
 3. Recall also the old open problem on coexistence of caustics with different rational rotation numbers for billiards other than ellipses.
 
 4. It would be interesting to establish analogous results for other billiard models that lead to twist maps of the cylinder, 
 in particular,  for Outer billiards, and the recently introduced Wire billiards \cite{BMT}. 
 
 5. In the light of the main result of the paper, 
 one would like to reconsider rigidity for continuous-time systems, like geodesic flows and Hamiltonian systems with a potential.

\section*{Acknowledgments}
It is a great pleasure to thank Serge Tabachnikov for encouraging and illuminating discussions we had for many years playing mathematical billiards.

We have discussed the Birkhoff conjecture with many mathematicians and we thank here all of them.

We are also grateful to anonymous referee for careful reading and improving suggestions.

The first author would like to thank the participants of the course "Geometry of billiards" delivered at Tel Aviv University
 in the Spring of 2020 and especially, Shvo Regavim, who gave a new proof of the parallelogram property for ellipses.

\section{\bf New generating function for the billiard map and the Twist condition}
Let $\gamma $ be a simple closed convex curve of positive curvature in $\mathbf R^2$. 
We fix the counterclockwise orientation on $\gamma$. We shall use the arclength parametrization $s$ as well as the parametrization by 
the angle $\psi$ formed by the outer unit normal $n$ to $\gamma$ with a fixed direction. These two parametrizations are related by $d\psi=k ds$, where $k$
is the curvature.

The natural phase space of the Birkhoff billiard inside $\gamma$ is the space $\mathbf A$ of all oriented lines that  intersect $\gamma$.
This space is topologically  a cylinder and we shall refer to it as the phase cylinder of $T$.
The billiard map $T$ acts on $\mathbf A$ by the reflection law in $\gamma$.
The phase cylinder carries a natural symplectic structure that can be described in two ways.

1.  Each oriented line is identified with the pair $(\cos\delta,s), \  \delta\in(0,\pi)$, where $\gamma(s)$ is the incoming point and $\delta$ is the angle between the line and the tangent $\gamma'(s)$. In these coordinates the symplectic form is 
$d\lambda$, where $\lambda=\cos\delta\, ds$. This is a standard symplectic form for the space $B^*\gamma$ of all (co-)tangent vectors of length $<1$, where $\cos\delta$ plays the role of momentum variable.
In this description, the boundaries of the phase cylinder are $\{\cos\delta=\pm1\}$.
	\begin{figure}[h]
	\centering
	\includegraphics[width=0.6\textwidth]{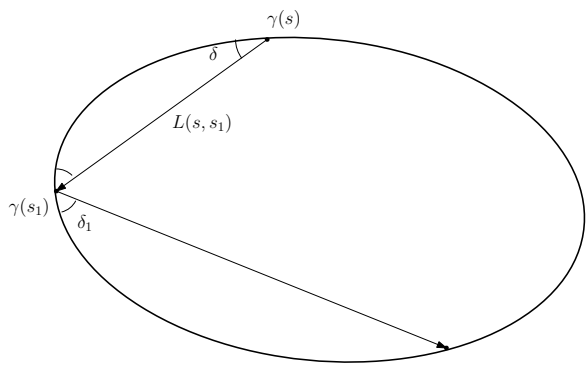}
	\caption{Generating function $L$ corresponding to the 1-form $\lambda$.}
	\label{fig:L}
\end{figure}

2.  The second way to get the same symplectic form (without any reference to $\gamma$) is to fix an origin in $\mathbf R^2$ 
and to introduce the coordinates $(p,\varphi)$ on the space of all oriented lines, so that $\varphi$ is the angle between the right unit normal to the line and the horizontal and $p$ is the signed distance to the line (see Fig. \ref{fig:S}). In this way the space of oriented lines is identified with $T^*S^1$. Moreover,  the standard symplectic form  $d\beta$\ with $\beta=pd\varphi$
coincides with the symplectic form described before. In this description $p$ plays the role of momentum variable. In the coordinates $(p,\varphi)$ the boundaries of the phase cylinder are $\{(p,\varphi):p= h(\varphi)\}, \{(p,\varphi):p=-h(\varphi +\pi)\}$, where throughout this paper  we shall denote by $h$ the support function of $\gamma$ with respect to $0$.

The billiard ball map is a symplectic map which can be defined with the help of generating functions.
In the first case this function is (see Fig. \ref{fig:L}) the length of the chord $L$ (this was used extensively by Birkhoff).
Namely, $$ T^*\lambda-\lambda=\cos\delta_1\,ds_1-\cos\delta\, ds=dL, \quad L(s,s_1)=|\gamma(s)-\gamma(s_1)|.
$$

For the second choice of the coordinates $(p,\varphi)$, the generating function was found first in \cite{BM} for the 2-dimensional case and then in \cite{B}  for higher dimensions (see \cite{BT} for further applications). This function $S$ is determined by the formulas:
$$
T^*\beta-\beta=p_1\,d \varphi_1-p\,d\varphi = dS, \quad S(\varphi ,\varphi _1 )=2h(\psi)\sin\delta,
$$
where $$\psi:=\frac{\varphi_1+\varphi}{2},\quad \delta:=\frac{\varphi_1-\varphi}{2}.$$
	\begin{figure}[h]
	\centering
	\includegraphics[width=0.8\textwidth]{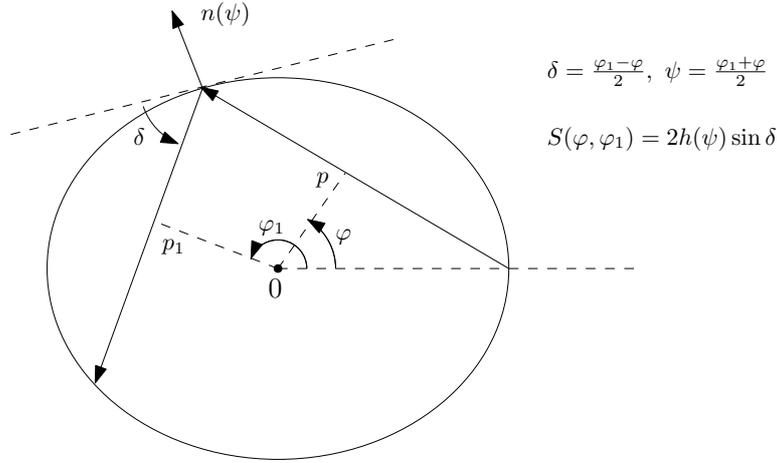}
	\caption{Generating function $S$ corresponding to the 1-form $\beta$}
	\label{fig:S}
\end{figure}
These formulas  mean that the line with coordinates $(p,\varphi)$ is mapped into the line $(p_1,\varphi_1)$ (see Fig. \ref{fig:S}) if and only if
\begin{equation}\label{generating}
\begin{split}
p&=-S_1(\varphi,\varphi_1)  =  h(\psi) \cos \delta-h'(\psi) \sin \delta,\\
p_1&=S_2(\varphi,\varphi_1)  =  h(\psi) \cos \delta+h'(\psi) \sin \delta.
\end{split}
\end{equation}

It is well known that the billiard ball map  satisfies the twist condition with respect to the symplectic coordinates $(\cos\delta,s)$, meaning that the cross-derivative $L_{12}(s,s_1)>0$. 

Here and below we use subindex $1$ and/or $2$ for the partial derivative with respect to the first or the second argument respectively. 

The geometric meaning of twist condition is that under the action of the differential of $T$ a vertical vector (tangent to the fibres $\{s=const\}$) deviates from the vertical .

Remarkably, the twist condition is  also satisfied for the generating function $S$ associated with the symplectic coordinates $(p,\varphi)$.
Indeed, one easily computes the cross derivative (see Proposition \ref{derivatives }):
$$
S_{12}(\varphi,\varphi_1)=\frac{1}{2}(h'' +h) \sin\delta=\frac{1}{2}\rho \sin\delta>0,
$$ where $\rho$ is the radius of curvature of $\gamma$.

Geometric meaning of this condition is that a beam of parallel lines ceases to be parallel after  reflection in $\gamma$.

Twist condition allows us to use Birkhoff's theorem on the  rotational invariant curves and the Aubry-Mather theory on variational properties of orbits.
\section{\bf Reduction of Theorem \ref{main1}}\label{sect:reduction}
In this section we shall present  two ways to reduce the statement of the main theorem to an even more geometric statement.
 This reduction is quite well understood by now. It was done for billiards in \cite{B0},   \cite{W}, \cite{B1} and further developed  in \cite{Arnaud}, \cite{Arnaud1}.

 For the billiard map $T:\mathbf A\rightarrow\mathbf A$ two vertical foliations arise naturally,  $\{s=const\}$ and $\{\varphi=const\}$, depending on which set of the symplectic coordinates we use. 

Let $\mathcal L$ be a sub-bundle of {\emph {oriented}} tangent lines to 
the domain $\mathcal A\subset \mathbf A$. 
We will say (following \cite{W}) that $\mathcal L$ is monotone with respect to the vertical foliation $\{s=const\}$ (respectively, $\{\varphi=const\}$), 
if the line $l(x)\subset
T_{x} \mathcal A$ is  non-vertical with respect to the vertical foliation $\{s=const\}$ (respectively, $\{\varphi=const\}$)  for all
$x\in\mathcal A$ and the orientation on the lines is given by $ds>0$ (respectively, $d\varphi>0$).

 The following lemma holds:
 \begin{lemma}\label{bundles}
 	Let $\mathcal L$ be a measurable monotone sub-bundle with respect to the vertical foliation $\{\varphi=const\}$, that is invariant under $T$.
 	Then  $\mathcal L$ is monotone also with respect to the vertical foliation $\{s=const\}$, and vice versa.
 \end{lemma}

Having this lemma we can state 
\begin{theorem}\label{main2}
	Let $\gamma\subset\mathbf R^2$ be a simple closed centrally-symmetric curve of positive curvature.  Let  $\mathcal A\subset \mathbf A$ be the domain bounded by the rotational  invariant curve $\alpha$ of $4$-periodic points and the boundary  $\{\delta=0\}$ of the phase cylinder. Suppose the  restriction of $T$ to $\mathcal A$ has a measurable monotone invariant sub-bundle. Then $\gamma $ is an ellipse.
\end{theorem}
Let us explain two ways of proving Theorem \ref{main1} using Theorem \ref{main2}. 
In fact, it can be concluded  from \cite{Arnaud1} that the existence of a monotone sub-bundle invariant under $T$ is equivalent to the continuous foliation property of Theorem \ref{main1}, called in \cite{Arnaud1} $C^0$-integrability. However, we will not use this equivalence and rather give a simpler arguments originally used in \cite{B0},  \cite{B1} and \cite{W}.

The first way of reduction (see \cite{W}) is based on the Birkhoff's theorem stating that each invariant rotational curve of $T$ is a graph of a Lipschitz function. Notice that the graph property holds with respect to both vertical foliations $\{s=const\}$ and $\{\varphi=const\}$. By the Lipschitz property we can define the line bundle of upper tangent lines to the graphs (corresponding to upper derivative). This line bundle is obviously measurable, monotone and invariant under $T$.
 Hence, Theorem \ref{main2} can be applied.
 Notice that if the foliation is assumed to be $C^1$-smooth, then $ \mathcal L$ is just the bundle of tangent lines to the  leaves and is continuous.

Another way of reduction (see \cite{B0}) is based on the Aubry-Mather theory for twist maps.
It follows from this theory (see the survey paper \cite{bangert}) that each orbit coming from an invariant graph is necessarily maximizing and in particular has no conjugate points. 
Moreover (implementing for billiards an idea of E. Hopf) one can construct \cite{B0}  in a measurable way, a non-vanishing Jacobi field along every billiard configuration and thus obtain a measurable, monotone invariant sub-bundle. This procedure of reduction is especially useful in establishing the following geometric fact.
\begin{corollary}
	Suppose  $\gamma$ is a centrally-symmetric convex closed curve of positive curvature. Assume there exists an  invariant curve $\alpha$ of rotation number $1/4$ that consists of $4$-periodic orbits.
	If $\gamma$ is not an ellipse, there always exist a point $x\in\mathcal A$ and a vertical tangent vector $v\in T_x\mathcal A$ such that for some positive integer $n$, the vector $DT^n(v)$ is again vertical (this  exactly means that the points $x$ and $T^n x$ are conjugate). Moreover, vertical vector in this statement can be understood with respect to each of the vertical foliations $\{s=const\}$ or $\{\varphi=const\}$.
	
	In particular, choosing the vertical foliation $\{\varphi=const\}$ we conclude that one can find a beam of parallel lines such that after $n$ reflections the beam becomes parallel (infinitesimally) again.
\end{corollary}
\subsection{Functions $\omega,\nu_1, \nu_{-1}$ and their properties}
In order to prove Lemma \ref{bundles} and Theorem \ref{main2} we shall apply the following general arguments. 

Consider a twist map $T$ of the cylinder 
\begin{equation}\label{twisth}
T:(p,q)\mapsto (p_1,q_1),\ 
\begin{cases}
p_1=H_2(q,q_1)\\
p=-H_1(q,q_1)
\end{cases}
\end{equation}
given by a generating function $H$ (can be one of the two discussed above) with the twist condition $H_{12}>0$. 

Let $\mathcal L$ be a measurable sub-bundle, monotone with respect to the vertical foliation $\{q=const\}$ and  invariant under $T$.
Lines of $\mathcal{L}$ are non-vertical and oriented by $dq>0$ (to the right).
Define the guiding vector field $u$ of $\mathcal{L}$ as follows.
Take a point $M=(p,q)$ and fix a positively oriented vector $$u(M)=\frac{\partial}{\partial q}+\omega(M)\frac{\partial}{\partial p}$$ on the line $l(M)$, where $\omega(M)$ is the slope of  $l(M)$. Let $$M_1=(p_1,q_1)=T(M),\quad M_{-1}=(p_{-1},q_{-1})=T^{-1}(M).$$
Then we can define two  functions
$$
\nu_1(M):=\pi_* DT(u(M))\quad and\quad  \nu_{-1}(M):=\pi_* DT^{-1}(u(M)),
$$where $\pi_*$ is the projection to the $q$ component (see Fig. \ref{fig:linebundle}).

 	\begin{figure}[h]
 	\centering
 	\includegraphics[width=0.8\textwidth]{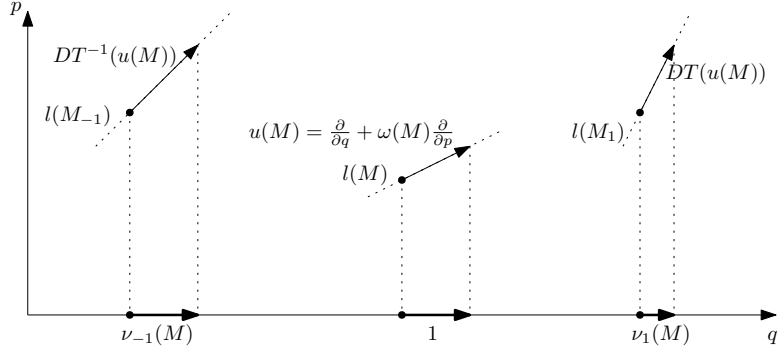}
 	\caption{Monotone invariant sub-bundle; oriented line $l(M)$ defines functions $\omega(M),\nu_1(M),\nu_{-1}(M). $}
 	\label{fig:linebundle}
 \end{figure}
  Moreover, we have 
  \begin{proposition}\label{positive}The following properties hold:
  	\begin{enumerate}[(a)]
  	\item $\omega, \nu_1 $, $\nu_{-1}$ are measurable functions,
  	\item $\nu_1, \nu_{-1}>0,$
  	\item $\nu_{-1}(T(M))=\nu_1(M)^{-1}.$
\end{enumerate}
  \end{proposition}
\begin{proof}Item (a) holds, since $\mathcal{L}$ is a measurable sub-bundle. Item (b), holds
	since the sub-bundle $\mathcal L$ is monotone and invariant (see Fig. \ref{fig:linebundle}). 
	In order to prove (c) notice that the vector field $u$, defined above, satisfies the equations (see Fig. \ref{fig:linebundle}):
%\begin{equation}
\begin{equation}\label{eq:nu}
	 {DT(u(M))=\nu_{1}(M)u(T(M))};\quad
{DT^{-1}(u(M))=\nu_{-1}(M)u(T^{-1}(M))}.
\end{equation}
Now replacing $M$ by $T(M)$ in the second relation of (\ref{eq:nu}) we get:
$$
DT^{-1}(u(T(M)))=\nu_{-1}(T(M))u(M).
$$
Applying $DT$ to both sides of this equation we have:
$$
u(T(M))=\nu_{-1}(T(M))DT(u(M)).
$$
Substituting here $DT(u(M))$ from the first identity of (\ref{eq:nu}) we have:
$$
u(T(M))=\nu_{-1}(T(M))\nu_{1}(M)u(T(M)).
$$
Thus,
$$
\nu_{-1}(T(M))\nu_{1}(M)=1,
$$completing the proof.
\end{proof}

Now we want to relate the functions $\omega$, and $ \nu_1$ in the following way.                
Differentiating (\ref{twisth}), we get for the differential $DT$ the relations
\begin{equation}\label{diff1}
\begin{cases}
dp_1=H_{12}(q,q_1)dq+H_{22}(q,q_1)dq_1\\
dp=-H_{11}(q,q_1)dq-H_{12}(q,q_1)dq_1.
\end{cases}
\end{equation}
And similarly
\begin{equation}\label{diff2}
\begin{cases}
dp=H_{12}(q_{-1},q)dq_{-1}+H_{22}(q_{-1},q)dq\\
dp_{-1}=-H_{11}(q_{-1},q)dq_{-1}-H_{12}(q_{-1},q)dq.
\end{cases}
\end{equation}
We evaluate the first relation of (\ref{diff2}) and the second relation of (\ref{diff1}) on the vector $u(p,q)=\frac{\partial}{\partial q}+\omega(p,q)\frac{\partial}{\partial p}$, and using the definitions of $\nu_1,\nu_{-1}$,
 we substitute there $$dq_{-1}=\nu_{-1}(p,q), \quad dq_1=\nu_1(p,q), \quad dp=\omega(p,q),\quad dq=1.$$ 

Thus, we can express $\omega (M)$ in two ways:
\begin{equation}\label{omega}
\begin{cases}
\omega(p,q)=H_{22}(q_{-1}, q)+H_{12}(q_{-1},q)\nu_{-1}(M),\\
\omega(p,q)=-H_{11}(q, q_1)-H_{12}(q,q_1)\nu_{1}(M).
\end{cases}
\end{equation}

In particular, it follows from (\ref{omega}), Proposition \ref{positive}(b),  and the twist condition $H_{12}>0$ that
\begin{equation}\label{estimate}
H_{22}(q_{-1},q)<\omega(p,q)<-H_{11}(q,q_1).
\end{equation} 
Shifting from  $M$ to $ T(M)$ in the first equation of (\ref{omega}) and using Proposition \ref{positive}(c) we get the relations:
\begin{equation}\label{omega1}
\begin{cases}
\omega(T(p,q))=H_{22}(q,q_1)+H_{12}(q,q_1)\nu_1^{-1}(M),\\
\omega(p,q)=-H_{11}(q, q_1)-H_{12}(q,q_1)\nu_{1}(M).
\end{cases}
\end{equation}
\subsection{Derivatives of generating function $S$}
Let us return now to the billiard ball map $T$. The derivatives of generating function $S$ can be immediately computed: 
\begin{proposition}\label{derivatives }
	The  second partial derivatives of $S $ are:
	\begin{align*}
	&
	S_{11}(\varphi,\varphi_1)=\frac{1}{2}(h''(\psi)-h(\psi))\sin \delta -h'(\psi)\cos \delta;&\\
&S_{22}(\varphi,\varphi_1)=\frac{1}{2}(h''(\psi)-h(\psi))\sin \delta +h'(\psi)\cos \delta;&\\
	&S_{12}(\varphi,\varphi_1)=\frac{1}{2}(h''(\psi)+h(\psi))\sin \delta ,&
	\end{align*}
	where $\psi:=\frac{\varphi_1+\varphi}{2},\quad \delta:=\frac{\varphi_1-\varphi}{2}.$
	\end{proposition}
In the sequel we shall work with the coordinates $(p,\varphi)$, the vertical foliation $\{\varphi=const\}$ and the function $\omega$
constructed above.
\subsection{Proof of Lemma \ref{bundles}}
 In one direction, let us assume that $\mathcal L$ is a monotone line bundle invariant under $T$. 
 Consider an oriented line incoming the billiard table at the point $\gamma(s)$ with the angle $\delta$.
 Let $\varphi$ is the angle between the right unit normal to the line and the horizontal, and $p$ is the signed distance from the origin to the line. 
The change of variables $(\cos\delta,s)_{\overrightarrow{F}} (p,\varphi )$ is given by the formulas
(exactly as $p_1$ in the second formula of (\ref{generating}))
$$
p=h(\psi)\cos\delta+h'(\psi)\sin\delta,\quad \varphi=\psi+\delta,
$$ where $\psi=\int k(s) ds$. 
Then the matrix of the differential  equals:

$$DF=
\begin{pmatrix}
	(	h(\psi)\sin\delta-h'(\psi)\cos\delta)\frac{1}{\sin\delta}& k(s)(	h'(\psi)\cos\delta+h''(\psi)\sin\delta) \\
-\frac{1}{\sin\delta}&k(s) 
\end{pmatrix}.
$$

The determinant of $DF$ is equal to $1$ (as it should be for a symplectic change of variables). The directing vector of the line $l(p,\varphi)$ of $\mathcal L$ in coordinates  $(p,\varphi ) $ has the components $(\omega,1)$ (as above, we denote $\omega$ the slope of the line in the coordinates $(p,\varphi)$).
We need to express this vector in the coordinates $(\cos\delta,s)$. Denote its components by $(a,b)$. 
$$
\begin{pmatrix}
a\\
b
\end{pmatrix}
=DF^{-1}\begin{pmatrix}
\omega\\
1
\end{pmatrix}.
$$
Then, using Proposition \ref{derivatives } we have

$$
b=\frac{1}{\sin\delta}(\omega+h(\psi)\sin\delta-h'(\psi)\cos\delta)=
$$
$$
=\frac{1}{\sin\delta}(\omega-S_{22}+S_{12})>0,
$$ 
by the twist condition and inequality (\ref{estimate}).
Therefore, the sub-bundle $\mathcal L$ is monotone also with respect to the vertical foliation $\{s=const\}$. The proof in the other direction uses the derivatives of the generating function $L$ (see for example \cite{B0}) and is completely analogous. $\Box$

\section{\bf Invariant curve of $4$-periodic orbits}
It is  well-known (see \cite{Connes-Z} for several proofs) that all Poncelet 4-gons for an ellipse
are  parallelograms.  Inspired by a proof of this fact by Shvo Regavim and by the recent paper \cite{BT} we can generalize this result from the case of an  ellipse  
to any centrally-symmetric billiard table as follows.

\begin{theorem}\label{4-periodic}
	Let $\gamma$ be a centrally-symmetric billiard table. Assume billiard ball map $T:\mathbf A \rightarrow\mathbf A$ has a continuous rotational invariant curve $\alpha=\{\delta=d(\psi)\}$ of rotation number $\frac{1}{4}$ consisting of $4$-periodic orbits of $T$. Then the following properties hold:
	
	 \begin{enumerate} [\rm (A)]
	 	\item Function $d(\psi)$ is  $\pi$-periodic and each billiard quadrilateral corresponding to invariant curve $\alpha$ is a parallelogram.
	 	
	 	\item The tangent lines to $\gamma$ at the vertices of the parallelogram form a rectangle.
	 
	 	\item  $0<d(\psi)<\pi/2,\quad d\left(\psi+\frac{\pi}{2}\right)=\frac{\pi}{2}-d(\psi).$
	 	
	 	\item The functions $d$ and $h$ satisfy the identities  $$\tan d(\psi)=\frac{h(\psi)}{h\left(\psi+\frac{\pi}{2}\right)}=-\frac{h'\left(\psi+\frac{\pi}{2}\right)}{h'(\psi)}, \ $$
	 	and $$h^2(\psi)+h^2\left(\psi+\frac{\pi}{2}\right)=R^2=const.$$
	 \end{enumerate}
  \end{theorem}

\begin{remark}
	When $\gamma$ is an ellipse with the semi-axes $a>b>0$  these properties are known. In particular, 
	$$h(\psi)=\sqrt{a^2\cos^2\psi+b^2\sin^2\psi},\quad   \cos2d= \frac{b^2-a^2}{a^2+b^2}\cos 2\psi,\quad  R^2=a^2+b^2.$$
	
\end{remark}
\begin{remark}
	It follows from Theorem \ref{4-periodic} item (D), that the orthoptic curve  associated with  $\gamma$ is a circle of radius  $R$, like in the case of an ellipse.
\end{remark}

\begin{remark}
	It follows from Theorem \ref{4-periodic} that the non-zero Fourier modes of the functions $d$ and $h^2$ (in addition to  the zero one) are of the form $e^{in\psi}, \ {\rm for}\ n\equiv 2\  (\rm mod 4)$. This observation gives the way to construct the examples of billiard tables with an invariant curve filled by $4$-periodic orbits.
\end{remark}

{\begin{remark}
	Let $\gamma$ be a centrally-symmetric billiard table.
	Suppose that the invariant curve $\alpha$ corresponds to a convex caustic, then this caustic is necessarily centrally-symmetric, since $d(\psi)$ is $\pi$-periodic.    Analogously, if the table $\gamma$ is symmetric with respect to $x,y$-axes, then this caustic is symmetric as well. These properties heavily rely on the fact that the invariant curve $\alpha$ consists of periodic orbits.
	We don't know if the result remains valid without this assumption (see \cite{AB} for a hypothetical counterexample and discussion for the case of rotation number $1/3$).
\end{remark}}
\begin{corollary}\label{relations1}
	Let $\gamma$ be a convex centrally-symmetric billiard table. Let $\alpha=\{\delta=d(\psi)\}\subset\mathbf A$ be an invariant curve consisting of $4$-periodic orbits. It then follows from Theorem \ref{4-periodic} item (D) that 
	$$
	h(\psi)= R\sin d(\psi) ,\quad h\left(\psi+\frac{\pi}{2}\right)=R\cos d(\psi), 
	$$
	for a positive constant $R$.
	\end{corollary}
\begin{corollary}
The explicit formulas of item (D) show that the invariant curve $\alpha$ is necessarily $C^2$-smooth, since the support function $h$ is $C^2$-smooth by assumption.
\end{corollary}
\begin{figure}[h]
	\centering
	\includegraphics[width=0.9\linewidth]{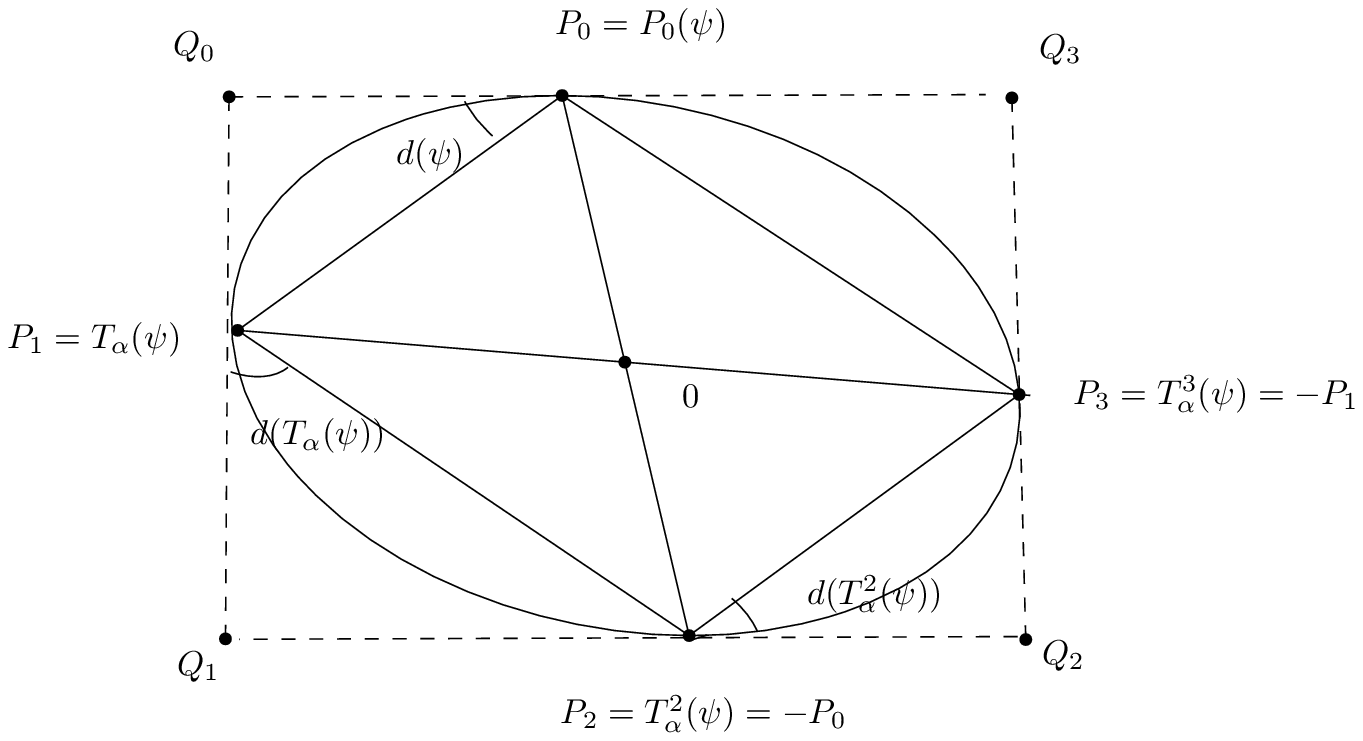}
	\caption{To Theorem \ref{4-periodic}}
	\label{fig:parallelo}
\end{figure}

\begin{proof}
We parametrize $\gamma$ and the invariant curve $\alpha=\{\delta=d(\psi)\}$ by the angle $\psi (\rm mod \ 2\pi)$.

Let us introduce the symmetry map: $S:\mathbf A\rightarrow \mathbf A$, which maps every oriented line intersecting $\gamma$ to the symmetric one with respect to $0$. Mappings $S$ and $T$ commute, since $\gamma$ is centrally-symmetric. Hence $S(\alpha)$ is an invariant curve of $T$ with the same rotation number.

In order to prove $\pi$-periodicity of $d(\psi)$, we shall use the following consequence of Aubry-Mather theory on the structure of minimal orbits with rational rotation number (see for example the survey paper \cite{bangert}, Theorem 5.8):

{\it If $\alpha$ is a rotational invariant curve of the cylinder twist map with rational rotation number which consists solely of periodic orbits, then any other rotational invariant curve with the same rotation number must coincide with $\alpha$.
}

{Indeed, it follows from this theory that all orbits lying on a rotational invariant curve are minimizing. Moreover, if the curve has a rational rotation number $\rho$ and all the orbits on this curve are periodic, then the set of these orbits coincides with the set $\mathcal M_{\rho}$ of the minimizers  with the rotation number $\rho$}.

Applying this corollary to the invariant curves $\alpha$ and $S(\alpha)$ we conclude $$\alpha=S(\alpha),$$ which immediately implies that $d(\psi)$ is $\pi$-periodic.

 We shall denote by $T_{\alpha}:\gamma\rightarrow\gamma$ the map induced by the invariant curve $\alpha$, and  by $\widetilde T_{\alpha}$ the lift  of $T_{\alpha}$ to $\mathbf R$. Thus, $\widetilde T_{\alpha}:\mathbf R\rightarrow \mathbf R$ is strictly monotone increasing and
\begin{equation}\label{T2pi}
\widetilde T_{\alpha}(\psi +2\pi)=\widetilde T_{\alpha}(\psi) +2\pi.
\end{equation}
Since the invariant curve consists of $4$-periodic orbits, 
\begin{equation}\label{T4}
\widetilde T_{\alpha}^{(4)}(\psi)=\psi+2\pi.
\end{equation}

It follows from $\pi$-periodicity of $d$ that
\begin{equation}\label{Tpi}
\widetilde T_{\alpha}(\psi+\pi)=\widetilde T_{\alpha}(\psi)+\pi.
\end{equation}
We claim that 
\begin{equation}\label{T2}
\widetilde T_{\alpha}^2(\psi)=\psi+\pi,
\end{equation}
meaning that the second iteration of a point is the centrally symmetric point.

\begin{figure}[h]
	\centering
	\includegraphics[width=0.7\linewidth]{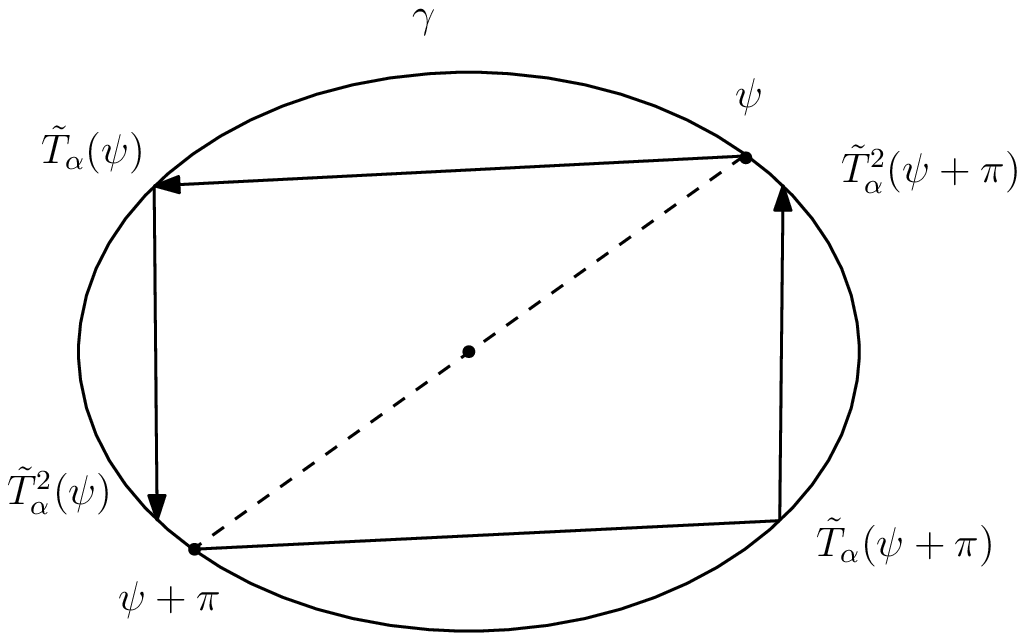}
		\caption{To inequality (\ref{contradiction})}
	\label{fig:inequality}
\end{figure}

Indeed, arguing by contradiction, suppose the  inequality 
\begin{equation}\label{contradiction}
\widetilde T_{\alpha}^2(\psi)<\psi+\pi
\end{equation} holds (the inequality $\widetilde T_{\alpha}^2(\psi)>\psi+\pi$ can be treated similarly).  Then by central symmetry (\ref{Tpi}) we get using (\ref{contradiction})
\begin{equation}\label{pi}
\widetilde T_{\alpha}^2(\psi+\pi)=\widetilde T_{\alpha}^2(\psi)+\pi<\psi+2\pi.
\end{equation}
Hence applying $\widetilde T_{\alpha}^2$ to (\ref{contradiction}) and using (\ref{pi}) we get,
$$
\widetilde T_{\alpha}^4(\psi)<\psi+2\pi,
$$
contradicting (\ref{T4}) (see Fig.\ref{fig:inequality}). This proves the claim (\ref{T2}).

Thus, for any $\psi$, formula (\ref{T2}) shows that the diagonals of the quadrilateral
 $\psi, T_{\alpha}(\psi), T_{\alpha}^2(\psi), T_{\alpha}^3(\psi)$ pass through the center $0$ and are centrally-symmetric. 
This proves property (A) in Theorem \ref{4-periodic}.

Next, we compute the sum of the exterior angles of parallelogram  $P_0P_1P_2P_3$ using the billiard reflection law (see Fig. \ref{fig:parallelo}):

$$
2({	2d(\psi)+2d(T_{\alpha}(\psi))})=2\pi\ \Leftrightarrow d(T_{\alpha}(\psi))+d(\psi)=\pi/2.
$$
Hence, $$T_{\alpha}(\psi)=\psi+\pi/2$$
proving both (B) and (C).

In order to prove (D) we follow the idea used in  \cite{BT} for the case of an ellipse.
Consider the line $P_0P_1$ and write its $p$-coordinate in two ways, using (\ref{generating}):
$$p(P_0P_1)=h(\psi) \cos {d}(\psi) +h'(\psi) \sin {d}(\psi)=$$
$$=h\left( {\psi+\frac{\pi}{2}} \right) \cos {d}\left( {\psi+\frac{\pi}{2}} \right) -h'\left( {\psi+\frac{\pi}{2}} \right) \sin {d}\left( {\psi+\frac{\pi}{2}} \right),$$
Analogously, since the line $P_1P_2$ is  centrally-symmetric with respect to $P_3P_0$, we have:
$$
p(P_1P_2)=p(P_3P_0)=h(\psi) \cos {d}(\psi) -h'(\psi) \sin {d}(\psi)=$$
$$=h\left( {\psi+\frac{\pi}{2}} \right) \cos {d}\left( {\psi+\frac{\pi}{2}} \right)+h'\left( {\psi+\frac{\pi}{2}} \right) \sin {d}\left( {\psi+\frac{\pi}{2}} \right).
$$
Summing and subtracting yields the identities:
$$
h\left( {\psi+\frac{\pi}{2}} \right)\cos{d}\left( {\psi+\frac{\pi}{2}} \right)=h(\psi)\cos{d}(\psi),
$$
$$
-h'\left( {\psi+\frac{\pi}{2}} \right)\sin{d}\left( {\psi+\frac{\pi}{2}} \right)=h'(\psi)\sin{d}(\psi).
$$
Using in these formulas $ d\left(\psi+\frac{\pi}{2}\right)=\frac{\pi}{2}-d(\psi)$ we get:
$$
\cot{d}(\psi)= \frac{h(\psi+\frac{\pi}{2})}{h(\psi)}=-\frac{h'(\psi)}{h'(\psi+\frac{\pi}{2})}.
$$
Also 
$$
h\left (\psi+\frac{\pi}{2}\right)  h'\left(\psi+\frac{\pi}{2}\right)+h(\psi)h'(\psi)=\left(h^2(\psi)+h^2(\psi+\frac{\pi}{2})\right)'=0.
$$
This completes the proof of Theorem \ref{4-periodic}.
	\end{proof}
\section{\bf Proof of Theorem \ref{main2}} We split the proof of Theorem \ref{main2} into three parts.
In the first part \ref{an-inequality}, we derive an integral inequality, in terms of function $d(\psi)$, that must hold under the assumptions of Theorem \ref{main2}.
In the second part \ref{reduction}, we show that for the curves satisfying the relations given by Theorem \ref{4-periodic} this inequality can be reduced to the converse of Wirtinger inequality, which therefore must be the equality.
In the third part \ref{completing}, we analyze the equality case in the Wirtinger inequality and
complete the proof of Theorem \ref{main2}.
\subsection{An inequality}\label{an-inequality}

Let $\gamma$ be a simple closed curve in $\mathbf R^2$ of positive curvature. Assume, the billiard map $T$ restricted to $\mathcal{A}$ has a measurable monotone invariant sub-bundle $\mathcal L$. We shall denote by $\omega(p,\varphi)$ the slope of the line $l(p,\varphi)$,  so that  $\frac{\partial}{\partial \varphi}+\omega\frac{\partial}{\partial p}$ is a positively oriented vector lying on $l(p,\varphi)$. We shall derive now an inequality   using the function $\omega$. 
Let us write the equations (\ref{omega1}) for the function $S$ and $\omega$, which are valid at every point $(p,\varphi)\in\mathcal A$ and its image $(p_1,\varphi_1)=T(p,\varphi))$:
\begin{equation}\label{omegaS}
\begin{cases}
\omega(T(p,\varphi))=S_{22}(\varphi,\varphi_1)+S_{12}(\varphi,\varphi_1)\nu_1^{-1}(p,\varphi),\\
\omega(p,\varphi)=-S_{11}(\varphi,\varphi_1)-S_{12}(\varphi,\varphi_1)\nu_{1}(p,\varphi).
\end{cases} 
\end{equation}
Upon multiplying the first equation by $p_1^2$, the second by $p^2$, and subtracting the results, we get
\begin{equation}\label{eq:S}\begin{split}
p_1^2&\omega(p_1,\varphi_1)-p^2\omega(p,\varphi))=\\
=&p^2S_{11}(\varphi,\varphi_1)+p_1^2S_{22}(\varphi,\varphi_1) +S_{12}(\varphi,\varphi_1)(p^2\nu_{1}(p,\varphi)+p_1^2\nu_1^{-1}(p,\varphi)).
\end{split}
\end{equation}

Using the twist condition and positivity of $\nu_1$ (by Proposition \ref{positive}(b)), we get from (\ref{eq:S}) the inequality:
\begin{equation}\label{ineq:S}
\begin{split}
p_1^2&\omega(p_1,\varphi_1)-p^2\omega(p,\varphi))\geq\\
\geq& p^2S_{11}(\varphi,\varphi_1)+p_1^2S_{22}(\varphi,\varphi_1) +2pp_1S_{12}(\varphi,\varphi_1).
\end{split}
\end{equation}

Notice that  from the estimate  (\ref{estimate}) and Proposition (\ref{derivatives }) it follows that function $\omega$ is bounded on $\mathcal A$:
$$
|\omega|<\max_{\mathcal A} \{|S_{11}|, |S_{22}|\}<K(\gamma),
$$ where $K(\gamma)$ depends only on $\gamma$. Therefore we can
integrate (\ref{ineq:S}) over $\mathcal{A}$ with respect to the invariant measure $d\mu=dpd\varphi$.

In order to perform the integration we compute the invariant measure as follows.

The symplectic form $dp \wedge d\varphi $  can be written using generating function (\ref{generating}):
$$
dp\wedge d\varphi=-d(S_1(\varphi,\varphi_1))\wedge d\varphi=S_{12}d\varphi\wedge d\varphi_1.
$$
Since $T$ is symplectic, the measure
$$d\mu= dp d\varphi=S_{12}d\varphi d\varphi_1,$$
 is invariant.
 Using the explicit formula for the second derivative (Proposition \ref{derivatives }) we compute:
 $$
 d\mu= S_{12}d\varphi d\varphi_1=
 $$
 $$
 =\left(\frac{1}{2}\rho(\psi)\sin\delta\right)d\varphi d\varphi_1=\left(\frac{1}{4}\rho(\psi)\sin\delta\right)d\psi d\delta,$$ where 
 $$\rho(\psi)=h''(\psi)+h(\psi), \quad  \psi:=\frac{\varphi_1+\varphi}{2},\quad \delta:=\frac{\varphi_1-\varphi}{2}.
 $$
 
Hence, integrating the inequality (\ref{ineq:S}) with respect to the invariant measure $d\mu$
we obtain:
$$
0\geq \int_{\mathcal A}[p^2S_{11}(\varphi,\varphi_1)+p_1^2S_{22}(\varphi,\varphi_1) +2pp_1S_{12}(\varphi,\varphi_1)]S_{12}\ d\varphi d\varphi_1.
$$
Substituting the  explicit expressions (\ref{generating}) for $p,p_1$ in terms of the generating function $S$
and  the explicit expression for $d\mu$ we get the inequality
\begin{equation*}
\begin{split}
0&\geq \int_{\mathcal A}\Big[(h(\psi) \cos \delta-h'(\psi) \sin \delta)^2\left(\frac{1}{2}(h''(\psi)-h(\psi))\sin \delta -h'(\psi)\cos \delta\right)+\\
& +(h(\psi) \cos \delta+h'(\psi) \sin \delta)^2\left(\frac{1}{2}(h''(\psi)-h(\psi))\sin \delta +h'(\psi)\cos \delta\right)+\\
& +(h^2(\psi) \cos^2 \delta-h'(\psi)^2\sin^2 \delta)(h''(\psi)+h(\psi))\sin\delta\Big](h''(\psi)+h(\psi))\sin\delta \ d\psi d\delta.
\end{split}
 \end{equation*}
 Simplifying the integrand we get

 \begin{equation}\label{eq-integral-delta}
 \begin{split}
 0&\geq \int_{0}^{2\pi}d\psi\int_{0}^{d(\psi)}d\delta\ \left(\cos^2\delta\sin^2\delta(h''h^2+3h(h')^2)(h+h'')-\right.
 \\
 &- \left.\sin^2\delta\  h(h')^2(h+h'')\right).
 \end{split}
 \end{equation}
 
 Here we used the fact that  in the coordinates $(\psi,\delta)$ the domain of integration is $\mathcal A=\{(\psi,\delta):\psi\in[0,2\pi],\  \delta\in[0,d(\psi)]\}$.
 
 Integrating in (\ref{eq-integral-delta}) with respect to  $\delta$ we obtain
\begin{equation}\label{krokodil1}\begin{split}
 0&\geq \int_{0}^{2\pi} \bigg[-h(h')^2(h+h'')\left(\frac{1}{2}d(\psi)-\frac{1}{4}\sin2d(\psi)\right)+\\
 &+(h''h^2+3h(h')^2)(h+h'')\left(\frac{1}{8}d(\psi)-\frac{1}{32}\sin4d(\psi)\right)\bigg]d\psi.
 \end{split}
\end{equation}

 By the assumption of central symmetry and Corollary \ref{relations1}, we know
 that $h(\psi),d(\psi)$ are $\pi$-periodic and related as follows:
 \begin{equation}\label{relations}
 \begin{cases}h=R\sin d,\\
 h'=R \cos d \ d', \\
 h''=R\cos d\  d''-R\sin d\ ( d')^2,\\
 d(\psi+\frac{\pi}{2})=\frac{\pi}{2}-d(\psi).
 \end{cases}
 \end{equation} Substituting these expressions into inequality (\ref{krokodil1}) we obtain an integral inequality that involves  the function $d(\psi)$ only. 
We perform this substitution in the next Subsection. 
\begin{remark}
	The factors $p^2,p_1^2$ in (\ref{eq:S}) were chosen so that   inequality (\ref{ineq:S}) becomes equality for the case of ellipses. This can be seen using the explicit form of the first integral for an ellipse, which is quadratic in momenta. We shall not dwell upon these details. 
\end{remark}
 \begin{remark}
	If one assumes $C^1$-regularity of the foliation by invariant curves in the hypotheses of Theorems \ref{main1}, then it is rather easy to get directly to the inequality (\ref*{ineq:S}) by suitable differentiation of the invariant graphs of the foliation, without the need for Section \ref{sect:reduction}. 
\end{remark}
\subsection{Reduction to the Wirtinger inequality}\label{reduction}
Let $U(\psi)
$ denote the integrand of (\ref{krokodil1}).

Since the curve $\gamma$ is centrally-symmetric 
$$
\int_0^{2\pi}U(\psi)d\psi=2 \int_0^{\pi}U(\psi)d\psi.
$$
By (\ref{krokodil1})  
\begin{equation}\label{U}
\int_{0}^{\pi}U \ d\psi\leq 0.
\end{equation}
We shall prove now the following
\begin{lemma}
$$
\int_0^{\pi}U(\psi)d\psi=\frac{\pi R^4}{512}\int_{0}^{\pi}((\mu'')^2-4(\mu')^2)\ d\psi,
$$where $\mu(\psi):=\cos (2d(\psi))$.
\end{lemma}

\begin{proof}
	To simplify the calculations,  we split $U(\psi)$ as
	$$
	U(\psi)=U_1(\psi)+U_2(\psi)+U_3(\psi),
	$$
	where
	\begin{align*}
		& U_1=\frac{1}{4}h(h')^2(h+h'')\sin(2d), &\\
		& U_2=-\frac{1}{32}h(h+h'')(3(h')^2+hh'')\sin(4d), &\\
		&U_3=\frac{1}{8}h(h+h'')\left(hh''-(h')^2\right)d.&
	\end{align*}

	Here and below in the proof we often omit the argument $\psi$ in the functions $h$ and $d$. 
	
	Using (\ref{relations}) we have
	\begin{align*}
		& U_1(\psi)=\frac{R^4}{8}(d')^2\sin^2(2 d)\left((1-(d')^2)\frac{\sin(2 d)}{2}+d''\cos^2(d)\right),&\\
		& U_2(\psi)=-\frac{R^4}{32}\sin(4d)\left((1-(d')^2)\sin^2(d)+
		\frac{\sin(2d)}{2}d''\right)&\\
		& \quad\quad\qquad \times\left((d')^2(4\cos^2(d)-1)+\frac{\sin(2d)}{2}d''\right),&\\
		& U_3(\psi)=\frac{R^4}{16}d\sin(d)\Big((1-(d')^2)\sin(d)+d''\cos(d)\Big)\left(d''\sin(2d)-2(d')^2\right).&
	\end{align*}
	
	\vspace*{0.3cm}
	We shall simplify the integral of $U$ in three steps. 
	\vspace*{0.3cm}
	
	{\underline{Step 1. Symmetrization.}}

	We perform the change  of the integration variable by the rule $\psi\rightarrow\psi+\frac{\pi}{2}$.  By (\ref{relations}), this intertwines $\sin (d)$ with $\cos(d)$ and changes the  sign of $d''$ and $\sin(4d)$.  Denote the changed integrand by $\hat U_j$
	
	\begin{align*}
		& \hat U_1(\psi)=\frac{R^4}{8}(d')^2\sin^2(2 d)\left((1-(d')^2)\frac{\sin(2 d)}{2}-d''\sin^2(d)\right),&\\
		& \hat U_2(\psi)=\frac{R^4}{32}\sin(4d)\left((1-(d')^2)\cos^2(d)-
		\frac{\sin(2d)}{2}d''\right)&\\
		& \quad\quad\qquad \times\left((d')^2(4\sin^2(d)-1)-\frac{\sin(2d)}{2}d''\right),&\\
		& \hat U_3(\psi)=\frac{R^4}{16}\left(\frac{\pi}{2}-d\right)\cos(d)\Big((1-(d')^2)\cos(d)-d''\sin(d)\Big)&\\
		&\quad\quad\quad\quad\times\left(-d''\sin(2d)-2(d')^2\right).&
	\end{align*}
	And denote the "symmetrized" integrand by 
	$$
	V_j:=\frac{1}{2}(U_j+\hat U_j).
	$$ 
	Then we have
	
	$$
	\int_0^{\pi}U_j(\psi)d\psi=\int_0^{\pi}\hat U_j(\psi)d\psi=\int_0^{\pi}V_j(\psi) d\psi,
	$$
	where $V_j$ can be written as:
	
	\begin{flalign*}
		&
		V_1(\psi)=\frac{R^4}{16}(d')^2\sin^2(2d)\left((1-(d')^2)\sin(2d)+d''\cos(2d)\right),
		&\\
		&
		V_2(\psi)=\frac{R^4}{64}\sin(4d)&\\
		&
		\qquad\qquad\ 	\times\left[\left((1-(d')^2)\cos^2(d)-\frac{\sin(2d)}{2}d''\right)
		\left((d')^2(4\sin^2(d)-1)-\frac{\sin(2d)}{2}d''\right)\right.&\\
		&
		\qquad\qquad\ \left.-\left((1-(d')^2)\sin^2(d)+\frac{\sin(2d)}{2}d''\right)
		\left((d')^2(4\cos^2(d)-1)+\frac{\sin(2d)}{2}d''\right)\right]&\\
		&
		\quad\quad\ =\frac{R^4}{128}\sin(4d)\left(2(d')^2((d')^2-1)\cos(2d) -d''(1+(d')^2)\sin(2d)\right),
		&\\
		&
		V_3(\psi) =\frac{R^4}{32}\left[d\sin(d)\left((1-(d')^2)\sin(d)+d''\cos(d)\right)\left(d''\sin(2d)-2(d')^2\right)\right.&\\
		&
		\qquad\qquad\left.-\cos(d) \left(\frac{\pi}{2}-d\right) \left((1-(d')^2)\cos(d) -d''\sin(d)\right)
		\left(d''\sin(2d)+2(d')^2\right)\right]&\\
		&
		\quad\quad\ 	=\frac{R^4}{32}\left(\pi\cos^2(d)-2d\cos(2d)\right)
		\left((d')^4-(d')^2\right)+
		\frac{\pi R^4}{128}\sin^2(2d)(d'')^2&\\
		&	
		\qquad\quad\quad +\frac{R^4}{64}\sin(2d)\left(2d-\pi\cos^2(d)\right)d''&\\
		&
		\qquad\qquad 	+\frac{R^4}{128}\sin(2d)\left(\pi(3+\cos(2d))-12 d\right)(d')^2d''.&\\
	\end{flalign*}
	{\underline{Step 2. Integration by parts.}} 
	\vspace*{0.5cm}
	
	We  apply  integration by parts in order to get rid of the second derivative $d''$. Notice that thanks to the $\pi$-periodicity of the integrands, the  off-integration terms vanish. We get
	\begin{align*}
		&
		\int_0^{\pi}V_1d\psi=\int_0^{\pi}\frac{R^4}{16}(d')^2\left(1-(d')^2\right)\sin^3(2d)d\psi\,&\\
		&\qquad\quad\quad\quad +\int_0^{\pi}\frac{R^4}{16} \sin^2(2d)\cos (2d) \ \left(\frac{1}{3} \frac{d} {d\psi} (d')^3\right)\ d\psi\\
		&
		\quad\quad\quad\quad=\int_0^{\pi}\frac{R^4}{16}(d')^2\left(1-(d')^2\right)\sin^3(2d)d\psi\,&\\
		&\qquad\quad\quad\quad -\int_0^{\pi}\frac{d}{d\psi}\left(\frac{R^4}{16} \sin^2(2d)\cos (2d)\right) \ \frac{(d')^3}{3}   
		\ d\psi
		=\int_0^{\pi}W_1(\psi)d\psi,
		&
	\end{align*}
	where
	$$
	W_1=\frac{R^4}{16} \left(\sin^3(2d) (d')^2-
	\left(4 \cos^2(2d)\sin(2d)+\sin^3(2d)\right)\frac{(d')^4}{3} \right).
	$$
	Next, to simplify $\int_0^{\pi}V_2d\psi$ we represent it as a sum of three terms and apply again the integration by parts
	
	\begin{align*}
		&
		\int_0^{\pi}V_2d\psi=\int_0^{\pi}\frac{R^4}{64}\cos(2d)\sin(4d) 
		\left((d')^4-(d')^2\right)d\psi&\\
		&
		\qquad\quad\quad\quad -\int_0^{\pi}\frac{R^4}{128} \sin(2d)\sin (4d) \ \left(\frac{d} {d\psi} (d')\right)\ d\psi&\\
		&
		\qquad\quad\quad\quad-\int_0^{\pi}\frac{R^4}{128} \sin(2d)\sin (4d) \ \left(\frac{1}{3} \frac{d} {d\psi}(d')^3\right)\       d \psi\\
		&
		\qquad\qquad=\int_0^{\pi}\frac{R^4}{64}\cos(2d)\sin(4d) 
		\left((d')^4-(d')^2\right)d\psi&\\
		&
		\qquad\quad\quad\quad+ \int_0^{\pi}\frac{d}{d\psi}\left(\frac{R^4}{128} \sin(2d)\sin (4d)\right) \  (d')\ d\psi&\\
		&
		\qquad\quad\quad\quad+\int_0^{\pi}\frac{d}{d\psi}\left(\frac{R^4}{128} \sin(2d)\sin (4d) \right)\ \left(\frac{(d')^3}{3} \right)\ d\psi
		=\int_0^{\pi}W_2(\psi)\ d\psi,
		&
	\end{align*}
	where
	\begin{align*}
		&	
		W_2=\frac{R^4}{64}\cos(2d)\sin(4d)((d')^4-(d')^2)
		+\frac{R^4}{64}\cos(2d)\sin(4d)(d')^2
		&\\
		&
		\qquad\ +\frac{R^4}{32}\sin(2d)\cos(4d)(d')^2+\frac{R^4}{192}\cos(2d)\sin(4d)(d')^4
		&\\
		&
		\qquad\  +\frac{R^4}{96}\sin(2d)\cos(4d)(d')^4
		&\\
		&
		\quad\ =\frac{R^4}{32}\sin(2d)\cos(4d)(d')^2+\frac{R^4}{96}\sin(2d)\left(2+3\cos(4d)\right)(d')^4.
		&
	\end{align*}	
	Finally,	similar calculations gives us
	\begin{align*}
		&
		\int_0^{\pi}V_3(\psi)=\int_0^{\pi}\frac{R^4}{32}\left(\pi\cos^2(d)-2d\cos(2d)\right)
		\left((d')^4-(d')^2\right) d  \psi&\\
		&\qquad\quad\quad\quad\ 
		+\int_0^{\pi}\frac{\pi R^4}{128}\sin^2(2d)(d'')^2d\psi&\\
		&\qquad\quad\quad\quad\ 
		+\int_0^{\pi}\frac{R^4}{64}\sin(2d)\left(2d-\pi\cos^2(d)\right)\ 
		\left(\frac{d} {d\psi} (d')\right)\ d\psi&\\
		&\qquad\quad\quad\quad\ 
		+ \int_0^{\pi}\frac{R^4}{128}\sin(2d)\left(\pi(3+\cos(2d))-12 d\right)\ 
		\left(\frac{1}{3} \frac{d} {d\psi}(d')^3\right)\ d\psi\\
		&
		\quad\quad \quad\quad =\int_0^{\pi}\frac{R^4}{32}\left(\pi\cos^2(d)-2d\cos(2d)\right)
		\left((d')^4-(d')^2\right) d  \psi&\\
		&\qquad\quad\quad\quad\ 
		+\int_0^{\pi}(\frac{\pi R^4}{128}\sin^2(2d)(d'')^2d\psi&\\
		&\qquad\quad\quad\quad\ 
		-\int_0^{\pi}\frac{d}{d\psi}\left(\frac{R^4}{64}\sin(2d)\left(2d-\pi\cos^2(d)\right)\right)\ 
		(d')\ d\psi&\\
		&\qquad\quad\quad\quad\ 
		- \int_0^{\pi}\frac{d}{d\psi}\left(\frac{R^4}{128}\sin(2d)\left(\pi(3+\cos(2d))-12 d\right)\right)\ 
		\frac{(d')^3}{3}\ d\psi
		=\int_0^{\pi}W_3(\psi)\ d\psi,
		&
	\end{align*}
	where	
	\begin{align*}
		&	
		W_3=\frac{\pi R^4}{128}\sin^2(2d)(d'')^2+
		\frac{R^4}{32}\left(\pi\cos^2(d)-2d\cos(2d)\right)
		\left((d')^4-(d')^2\right)&\\
		&\quad\  \quad 
		-\frac{R^4}{32}\cos(2d)(2d-\pi\cos^2(d))(d')^2-
		\frac{R^4}{32}\sin(2d)(1+\pi\cos(d)\sin(d))(d')^2&\\
		&\quad\  \quad 
		-\frac{R^4}{64}\cos(2d)\left(\pi(3+\cos(2d))-12d\right)\frac{(d')^4}{3}+
		\frac{R^4}{64}\sin(2d)\left(\pi\sin(2d)+6\right)\frac{(d')^4}{3}	&\\
		&\  \quad 
		=\frac{\pi R^4}{128}\sin^2(2d)(d'')^2&\\
		&\quad\  \quad 
		-\frac{R^4}{32}\left(\sin(2d)+\pi\left(\cos^2(d)(1-\cos(2d))+\frac{1}{2}\sin^2(2d)\right)    \right)(d')^2 &\\
		&\quad\  \quad 
		+ \frac{R^4}{32}\left(\sin(2d)+\pi\left(\cos^2(d)-\frac{1}{6}\cos(2d)(3+\cos(2d))+\frac{1}{6}\sin^2(2d)\right)\right)(d')^4
		&\\
		&\ \quad 
		=\frac{\pi R^4}{128}\sin^2(2d)(d'')^2-\frac{R^4}{32}\sin(2d)
		\left(1+\pi\sin(2d)\right)(d')^2&\\
		&\quad\quad\  
		-\frac{R^4}{192}\left(\pi\cos(4d)-3\left(\pi+2\sin(2d)\right)\right)(d')^4.
		&
	\end{align*}

	We conclude that $\int_0^{\pi}U(\psi)d\psi=\int_0^{\pi}W(\psi)d\psi,$
	where 
	
	\begin{align*}
		&
		W=W_1+W_2+W_3=\frac{\pi R^4}{128}\sin^2(2d)(d'')^2&\\
		&
		\quad\quad +\frac{R^4}{32}\left(2\sin^3(2d)+\cos(4d)\sin(2d)-\sin(2d)(1+\pi\sin(2d))\right)(d')^2&\\
		&\quad\quad 
		+\frac{R^4}{48}\left(\frac{\sin(2d)}{2}(2+3\cos(4d))-
		(4\cos^2(2d)\sin(2d)+\sin^3(2d))\right.&\\
		&\quad\quad 
		-\left.\frac{1}{4}\left(\pi\cos(4d)-3(\pi+2\sin(2d))\right)\right)(d')^4&\\
		&\quad=-\frac{\pi R^4}{32}\sin^2(2d)(d')^2+
		\frac{\pi R^4}{192}\left(3-\cos(4d)\right)(d')^4+\frac{\pi R^4}{128}\sin^2(2d)(d'')^2.
		&
	\end{align*} 
	$$
	$$
	
	{\underline{Step 3. Change of the function.}} 
	\vspace*{0.5cm}
	
	For the last step in the proof, we introduce the new function $$
	\mu(\psi)=\cos(2d(\psi)).
	$$
	We want to express $W$
	as a function of $\mu$.
	We have
	$$
	\mu'=-2\sin(2d) d',\ \ \mu''=-4\cos(2d)(d')^2-2\sin(2d)d'',
	$$
	and hence
	$$
	\sin^2(2d)(d'')^2=\frac{1}{4}(\mu'')^2-4\cos^2(2d)(d')^4-2\sin(4d)(d')^2d''.
	$$
	Substituting this in the expression for $W(\psi)$, we get
	$$
	W=-\frac{\pi R^4}{64}\sin(4d)(d')^2d''-
	\frac{\pi R^4}{384}\left(12\sin^2(2d)(d')^2+8\cos(4d)(d')^4-3(\mu'')^2/4\right).
	$$
	Another integration by parts yields
	\begin{align*}
		&
		\int_0^{\pi}W(\psi)d\psi=
		&\\
		&\quad-\int_0^{\pi}\frac{\pi R^4}{64}\sin (4d)\ \left( \frac{1}{3} \frac{d} {d\psi}(d')^3\right)\ d\psi
		&\\
		&
		\quad -\int_0^{\pi}\frac{\pi R^4}{384}\left(12\sin^2(2d)(d')^2+8\cos(4d)(d')^4-3(\mu'')^2/4\right)d\psi&\\
		&
		=\int_0^{\pi}\frac{d}{d\psi}\left(\frac{\pi R^4}{64}\sin (4d)\right)
		\frac{(d')^3}{3}  d\psi&\\
		&\quad-\int_0^{\pi}\frac{\pi R^4}{384}\left(12\sin^2(2d)(d')^2+8\cos(4d)(d')^4-3(\mu'')^2/4\right)d\psi&\\
		&
		=\int_0^{\pi}P(\psi)d\psi,
		&\\
	\end{align*}
	where
	$$
	P=\frac{\pi R^4}{512}((\mu'')^2+8(\cos(4d)-1)(d')^2)=\frac{\pi R^4}{512}((\mu'')^2-4(\mu')^2).
	$$
	Thus we proved 
	$$
	\int_0^{\pi}U(\psi)d\psi=\int_0^{\pi}P(\psi)d\psi=\frac{\pi R^4}{512}((\mu'')^2-4(\mu')^2).
	$$
	This completes the proof of lemma.
\end{proof}
\subsection{\bf Completing the proof of Theorem \ref{main2}}\label{completing}
Now we are in position to finish the proof of Theorem \ref{main2}. By definition, $\mu$ is a $\pi$-periodic function and therefore Wirtinger's inequality can be applied to $\mu'$ (note the factor $4$ due to period $\pi$ and not $2\pi$).  Thus, we have:
$$
\int_{0}^{\pi}U\ d\psi=\int_{0}^{\pi}P\ d\psi=\frac{\pi R^4}{512}\int_{0}^{\pi}((\mu'')^2-4(\mu')^2)d\psi\geq 0.
$$ 
But this is exactly opposite to the inequality (\ref{krokodil1}) we started with.
Hence, we are in the equality case in the Wirtinger's inequality and therefore
$$
\mu(\psi)=\cos(2d(\psi))=a\cos(2\psi)+b\sin(2\psi)+c.$$
We can rewrite this in the form
\begin{equation}\label{A}
\cos(2d(\psi))=A\cos(2\psi+\alpha) +c, \ A\geq 0.
\end{equation}
 
 Shifting $\psi$, if needed, we can assume $\alpha=0$.  
 
 Now we claim that $c$ must vanish.
 Indeed, by Theorem \ref{4-periodic}, 
 $$
 R^2\cos(2d(\psi))=R^2(\cos^2 d(\psi)-\sin^2 d(\psi))=h^2(\psi+\pi/2)-h^2(\psi).
 $$
 Notice that the last expression  has average $0$ over the period $[0,\pi]$. Therefore, in (\ref{A}) $c=0$.
 
 Moreover, $A\in [0,1)$, since (\ref{A}) holds for all values of $\psi$ ($A=1$ is not allowed since $d$ does not vanish).
 Consequently, $$
 1-2\sin^2 d(\psi)=A\cos2\psi, \quad A\in [0,1).
 $$
Hence, by Theorem \ref{4-periodic},  we obtain 
\begin{align*}
	&
	h^2 (\psi) =R^2\sin^2 d(\psi)=\frac{R^2}{2}(1-A\cos 2\psi)=
	&\\
	&
	\qquad\ \  =\frac{R^2(1-A)}{2}\cos^2\psi+\frac{R^2(1+A)}{2}\sin^2\psi.
	&
\end{align*}
The square root of the last expression is the support function of an ellipse, for every $A\in[0,1)$ (the case $A=0$ gives a circle).
This completes the proof of Theorem \ref{main2}.

\end{document}